\documentclass[11pt,epsfig]{article}

\usepackage{amsmath,amssymb,amsthm}
\usepackage{graphics}

\usepackage{wrapfig}
\usepackage{hyperref}

\usepackage{tikz}

\newtheorem{theorem}{Theorem}[section]
\newtheorem{lemma}[theorem]{Lemma}

\theoremstyle{definition}

\newtheorem{remark}[theorem]{Remark}

\newcommand{\rank}{\mathop{\rm rank}\nolimits}
\newcommand{\conf}{\mathop{\rm Conf}\nolimits}

\title {Log--optimal $(d+2)$-configurations in $d$--dimensions}

\author {Peter D. Dragnev and Oleg R. Musin}

\usepackage{fancyhdr}
\begin{document}
\date{}
\maketitle


%

\begin{abstract} We enumerate and classify all stationary logarithmic configurations of $d+2$ points on  the unit sphere in $d$--dimensions. In particular, we show that the logarithmic energy attains its local minima at configurations that consist of two orthogonal to each other regular simplexes of cardinality $m$ and $n$. The global minimum occurs when  $m=n$ if $d$ is even and $m=n+1$ otherwise. This characterizes a new class of configurations that minimize the logarithmic energy on $\mathbb{S}^{d-1}$ for all $d$. The other two classes known in the literature, the regular simplex ($d+1$ points on $\mathbb{S}^{d-1}$) and the cross polytope ($2d$ points on $\mathbb{S}^{d-1}$), are both universally optimal configurations.
\end {abstract}

\medskip

\noindent {\bf Keywords:} Thomson's problem, Riesz potential, logarithmic energy, optimal configurations

\vskip 2mm

\noindent {\bf Mathematics Subject Classification}: {{\it Primary} 74G05, 74G65; {\it Secondary} 31B15, 31C15}

\medskip


\section{Introduction and main result}
\setcounter{equation}{0}

 Let $X=\{x_1,\ldots, x_N\}$ be a set of distinct points (unit vectors) on the unit sphere ${\mathbb S}^{d-1}$ in ${\mathbb R}^{d}$. Configurations that minimize the {\em logarithmic energy}
 
\begin{equation}\label{LogEn}
E_{\rm log} (X):= \sum_{1\leq i \not= j\leq N} \log
\frac{1}{|{x}_i-{ x}_j|}=-(1/2)\sum_{1\leq i \not=j\leq N} \log
(1-{x}_i\cdot{ x}_j) -\frac{N(N-1)\ln{2}}{2},
\end{equation}
are called {\em log-optimal}. More generally, a configuration is called {\em $h$-optimal} for a potential interaction $h:[-1,1)\to \mathbb{R}$, if it minimizes the {\em $h$-energy}
\begin{equation}\label{hEn}
E_{h} (X):= \sum_{1\leq i \not=j\leq N} h({x}_i\cdot{ x}_j).
\end{equation}
The {\em Newton potential} ($h(t)=(1-t)^{-d/2+1}$), and more generally the {\em Riesz potential} ($h(t)=(1-t)^{-s/2}$) , as well as the {\em Gaussian potential} ($h(t)=e^{\alpha t}$, $\alpha>0$) have been well studied in the literature (see \cite{SK}). The {\em logarithmic potential} $-\log (1-t)$ is the limiting case of the Riesz potential as $s\to 0$. All of these potentials are absolutely monotone potentials, i.e. $h^{(k)}(t)\geq 0$, for all $k=1,2,\dots$. The {\em regular simplex} ($N=d+1$) and the {\em cross polytope} ($N=2d$) are the only known classes of configurations that minimize the logarithmic energy for all $d$; actually, they are {\em universally optimal configurations}, namely they minimize the energy for all absolutely monotone potentials $h$ (see \cite[Table 1]{CK}). Another (infinite) class of universally optimal configurations is the so-called {\em isotropic spaces}, for which $d=q(q^2+q+1)$ and $N=(q+1)(q^3+1)$, where $q$ is a power of a prime number. All other known optimal configurations in the literature, even when the interacting potential $h$ is fixed, have particular values of the dimension $d$ and the cardinality $N$.

While the original problem of finding log-optimal configurations on the sphere, sometimes referred to as {\em Whyte's problem} (see \cite{W}), was posed in 1952, few advances have been made throughout the years. That the regular simplex is a log-optimal configuration follows from the classical arithmetic-geometric mean inequality. Kolushev and Yudin \cite{KY}, using analytic methods derived in 1997 that the cross-polytope (the $2d$ intersection points of the coordinate axes and the unit sphere) minimizes the logarithmic energy. In 1996 Andreev \cite{A} proved that the regular icosahedron is a log-optimal configuration. Subsequently, in 2007 Cohn and Kumar \cite{CK} showed all these to be universally optimal configurations (ones that minimize all absolutely monotone potentials). The first non-universally optimal case of $d+2$ points on $\mathbb{S}^{d-1}$ for $d=3$ was resolved in 2002 (see \cite{DLT}) and the cases $d=4$ and $d=5$ were derived in 2016 (see \cite{D}). 

Note that all partial results have been focused on finding the global minima. The goal of this article somewhat more general, namely to classify all local minima for the logarithmic energy for the class of $N=d+2$ points on  $\mathbb{S}^{d-1}$, $d\geq 2$, and in particular, determine the log-optimal energy configuration for this class. The following is our main theorem.

\begin{theorem}\label{thmLogOpt} Up to orthogonal transform, every local minimum of the logarithmic energy  $E_{\log} (X)$ of $d+2$ points on $\mathbb{S}^{d-1}$ consists of two regular simplexes of cardinality $m\geq n>1$, $m+n=d+2$, such that these simplexes are orthogonal to each other. The global minimum occurs when $m=n$ if $d$ is even and $m=n+1$ otherwise. 
\end{theorem}

The theorem is derived following a careful analysis of non-degenerate stationary configurations. While inspired by \cite{D}, our approach in this article is new and allows us to establish much stronger necessary conditions for stationarity (see Theorems \ref{thmMain} and \ref{thmPyramid}). As pointed in Remark \ref{D}, that the number of orthogonal simplexes in Theorem \ref{thmMain} is two, a byproduct of Theorem \ref{thmA}. For degenerate stationary configurations, Theorem \ref{thmDeg} shows that the $h$-energy may be decreased whenever $h$ is strictly convex potential function, including in the logarithmic case.

Note that for $d$ even the log-optimal configuration in Theorem \ref{thmLogOpt} is a {\em two-distance set} (see \cite{Mu} and references therein) that is the {\em two-design} introduced by Mimura \cite{Mi}. We also draw the reader's attention to a remarkable connection with the classification of best packing configurations of $d+k$, $1\leq k\leq d$ points on $\mathbb{S}^{d-1}$ found by W. Kuperberg in \cite{K}. In particular, his classification implies that any best packing configurations of $d+2$ points will split into two orthogonal simplexes, not necessarily regular, but with minimal distance at least $\sqrt{2}$. It is easy to see that the local minima above minimize the logarithmic energy among such best packing configurations. Kuperberg-type theorems for two--distance sets are considered in \cite{Mu}. We finally point out the connection with Steven Smale's  $7^{\rm th}$ problem \cite{Sm} asking for generating in polynomial time nearly log-optimal configurations on $\mathbb{S}^2$ for large $N$.

In the next section we classify the stationary configurations and deal with the cases that don't lead to local minima. In Section 3 we introduce some auxiliary results utilized in Section 4 to prove  the results about stationary configurations that are saddle points. The proof of the main theorem is presented in Section 5. In Section 6  we derive the Morse index for all stationary configuration of five points on $\mathbb{S}^2$ and list some related open problems and future plans of research.

\section{Stationary Configurations of $d+2$ points on $\mathbb{S}^{d-1}$}

In this section we completely classify the stationary configurations of $d+2$ points on $\mathbb{S}^{d-1}$. We call a configuration $X$ {\em non-degenerate} if span$(X)=\mathbb{R}^d$ and {\em degenerate} otherwise. 

\begin{theorem} \label{thmMain} Let $N=d+2$ and $X=\{x_1,\ldots, x_N\}$ be a non-degenerate stationary logarithmic configuration on ${\mathbb S}^{d-1}$.  Suppose there is no point $x \in X$  that is equidistant to all other points in $X$. Then $X$ can be split into two sets such that these sets are vertices of two regular orthogonal simplexes with the centers of mass in the center of\, ${\mathbb S}^{d-1}$. 
\end{theorem}

\begin{remark} \label{D} This theorem strengthens significantly the characterization theorem \cite[Theorem 1.5]{D}, which asserts that a stationary configuration is either degenerate; has a vertex equidistant to all others; or that every vertex has a mirror related partner, i.e. another vertex, such that the perpendicular bisector hyperplane of the segment formed by the two vertices contains all other points of the configuration. The mirror relation as an equivalence relation splits the points in a non-degenerate stationary configuration that has no vertex equidistant to all other vertices into equivalence classes that form regular simplexes. Theorem \ref{thmMain} states that these simplexes are only two. This along, together with \cite[Lemma 3.2]{D} implies the global minimum part of Theorem \ref{thmLogOpt}.
\end{remark}

In the process of classifying all local minima for the energy, we need to eliminate the other cases. We first consider degenerate stationary configurations. While there are such configurations that are global minimizers of energy among all configurations confined to their spanning subspace  (say a regular pentagon on the Equator of $\mathbb{S}^2$), the next theorem (a generalization of \cite[Theorem 1.6]{D}) shows that for any strictly convex potential function $h$, the $h$-energy (see \eqref{hEn}) of a degenerate configuration with cardinality $N\geq d+2$ can be strictly decreased by a small perturbation, and hence may not be a local minimum.

\begin{theorem}\label{thmDeg}
Let $X$ be a degenerate configuration, $N\geq d+2$, and $h : [-1,1]\to \mathbb{R}$ be a strictly convex potential function. Then there exists a continuous perturbation that decreases the $h$-energy $E_h(X)$.
\end{theorem}

Next, we focus on configurations that are not degenerate, but have a vertex, say the North Pole $x_N$, that is equidistant to all other vertices $x_j$. We shall denote such configurations with $\{ 1,N-1 \}$. Then the vertices $\{ x_1,\dots, x_{N-1} \}$ are lying on a hyperplane in the Southern hyper-hemisphere at height $-1/(N-1)$. By projecting these vertices to the Equatorial hyperplane and normalizing to become unit vectors, we reduce the configuration to $d+1$ points on $S^{d-2}$ that form a non-degenerate stationary (w.r.t logarithmic energy) configuration. This configuration may have a vertex that is equidistant to all others, we shall denote such a case as $\{1,1,N-2\}$. As for $4$ points on $\mathbb{S}^1$ the only stationary configuration is the two orthogonal simplexes split (diagonals of a square), this process will stop with two orthogonal simplexes case. The following theorem sheds light on this case.

\begin{theorem}\label{thmPyramid} A non-degenerate stationary log-energy configuration of type $\{ 1,1,\dots,k,l \}$, where $1+1+\dots+k+l=d+2$ is a saddle point. Moreover, there is a continuous perturbation that decreases the logarithmic energy of the $\{1,k,l\}$ part of the configuration to either $\{k+1,l\}$ or $\{k,l+1\}$. Subsequently, with a sequence of such perturbations, one can reach a local minimum as described in Theorem \ref{thmLogOpt}. 
\end{theorem}

\section{Auxiliary results}

Utilizing Lagrange multipliers to the constrained minimization of \eqref{LogEn} we show that for any {\em stationary configuration} $X$ the following vector equations (also referred to in the literature as {\em force equations}) hold true
$$
\sum_{j\not= i} \frac{{ x}_i- { x}_j}{r_{i,j}} = \lambda_i
{ x}_i\ \ \ i=1,\dots, N,
$$
where $ r_{ij}:=1-x_i\cdot x_j$. Taking inner product of both sides with $x_i$ one obtains $\lambda_i=N-1$, $i=1,\dots,N$.  Therefore, 
\begin{equation}
\label{VecEqs} 
\sum_{j\not= i} \frac{{ x}_i- { x}_j}{r_{i,j}} = (N-1)\,{ x}_i,\ \ \ i=1,\dots, N.
\end{equation}

Summing \eqref{VecEqs} implies that the centroid of a stationary configuration $X$ lies at the origin and that for all $i=1,\ldots,N$ we have 
\begin{equation}\label{SumDist}
\sum\limits_{j}{r_{ij}}=N. 
\end{equation}
  Let 
 $$
B=\left(b_{ij}\right),\quad  b_{ij}:=\frac{1}{ r_{ij}}, \quad b_{ii}:=N-1- \sum\limits_{j\ne i}{ b_{ij}},
 $$
  $$
A=\left(a_{ij}\right), \, \mbox{ where }  a_{ij}:=c-b_{ij}, \quad c:=\frac{N-1}{N}. 
 $$


 \begin{lemma} \label{lemma21} Let  $X=\{x_1,\ldots, x_N\}$ be a  stationary logarithmic configuration on  ${\mathbb S}^{d-1}$ that is non-degenerate $($span$(X)=\mathbb{R}^d)$. Then  
 $$
\rank(A)\leq N-d-1, \qquad \sum\limits_{j=1}^N{a_{ij}}=0, \quad i=1,\ldots,N. 
 $$
\end{lemma}

\begin{proof} Let ${\rm X}:=[x_1,\dots,x_N]^T$. The force equations (3) imply that 
$$
\sum\limits_{j=1}^N{b_{ij}x_j}=0, \quad  \sum\limits_{j=1}^N{b_{ij}}=N-1.
$$
In other words, $B{\rm X}=0$ and $B{\bf 1}=(N-1){\bf 1}$, where $\bf 1 $ denotes the $N$-dimensional column-vector of ones. As $X$ is non-degenerate, we have rank ${\rm X}=d$. Therefore, the column-vectors of ${\rm X}$ are linearly independent. As $\bf 1$ is eigenvector of $B$ with an eigenvalue of $N-1$ it is linearly independent to the columns of  ${\rm X}$ (eigenvectors with eigenvalue $0$). The lemma follows from the rank-nullity theorem applied to $A[{\rm X},{\bf 1}]=0$.
\end{proof}

The following lemma elaborates on the case when $N=d+2$.

\begin{lemma}\label{lm2} Let $N=d+2$ and $X=\{x_1,\ldots, x_N\}$ be a non-degenerate stationary logarithmic configuration on ${\mathbb S}^{d-1}$.  Without loss of generality we may assume that $a_{1i}\ge0$ for $i=1,\ldots k$ and $a_{1i}<0$ for $i=k+1,\ldots N$. Let 
$$
 a_i=\sqrt{a_{ii}}, \, i=1,\ldots k; \; a_i=-\sqrt{a_{ii}}, \, i=k+1,\ldots N. 
$$
Then 
$$
a_{ij}=a_i\,a_j, \quad a_1+\ldots+a_N=0, 
$$
$$
c-a_ia_j\ge \frac{1}{2}, \; \mbox{ for all } \: i\ne j,
$$
\begin{equation} \label{req}
 \sum\limits_{j\ne i} {\frac{1}{c-a_ia_j}}=N, \; i=1,\ldots,N. 
\end{equation}
\end{lemma} 
\begin{proof} We first observe that if $N=d+2$, then $\rank(A)=1$. Indeed, $\rank (A)=0$ yields that all mutual distances are equal, which is impossible. 

Since $A$ is a symmetric matrix of rank 1, $a_{ij}=a_i\,a_j$ for all $i, j$. Lemma \ref{lemma21} implies that  for all $i$ we have 
$$\sum\limits_j{a_{ij}}=a_i(a_1+\ldots+a_N)=0.$$ 
Since all $a_i$ cannot be $0$, we have $a_1+\ldots+a_N=0.$

By definitions we have $a_{ij}=c-1/r_{ij}$, i.e. 
$$
r_{ij}=\frac{1}{c-a_{ij}}=\frac{1}{c-a_ia_j}, \; i\ne j. 
$$
Since $r_{ij}\le 2$, we have 
$$
c-a_ia_j\ge \frac{1}{2}. 
$$
It is easy to see that \eqref{SumDist} implies \eqref{req}. 
\end{proof}

Note that if $a_i=0$ then the $i$-th row and  $i$-th column in the matrix $A$ are zero. Therefore, $x_i$ is equidistant to all other points $x_j$ and 
$$
r_{ij}=\frac{N}{N-1}, \quad j=1,...,i-1,i+1,\ldots,N. 
$$
Thus, if a configuration has no point that is equidistant to all others, then $a_i \not= 0$ for all $i=1,\dots, N$. 

The following theorem is the main in this section. 

\begin{theorem} \label{thmA} Let $a_1,\ldots,a_N$ be real numbers that satisfy the following assumptions 
$$
a_1\ge\ldots\ge a_k>0>a_{k+1}\ge\ldots\ge a_N, \quad a_1+\ldots,+a_N=0, 
$$ 
$$
 \sum\limits_{j\ne i} {\frac{1}{c-a_ia_j}}=N, \; i=1,\ldots,N, \quad c-a_ia_j>0, \; \mbox{ for all } \; i\ne j, \; \mbox{ where } \; c:=\frac{N-1}{N}.  
$$
Then 
$$
a_1=...=a_k, \quad a_{k+1}=...=a_N. 
$$
\end{theorem}

First we prove  two technical Lemmas. 

\begin{lemma}\label{lm3} Suppose  $a_1,\ldots,a_N$  are as in Theorem \ref{thmA}. Then for all $i=1,\ldots,N$ we have
\begin{equation}\label{T_i}
T_i :=\sum\limits_{j\not= i}{\frac{c-a_j^2}{c-a_ia_j}}=N-2. 
\end{equation}
\end{lemma}
\begin{proof}

Let 
$$ 
Q_i:=\sum\limits_{j\ne i} {\frac{1}{c-a_ia_j}}.
$$
Then by the assumption $Q_i=N$ for all $i$. 

Let 
$$
R_i:=\sum\limits_{j\ne i} {\frac{a_j}{c-a_ia_j}}. 
$$
Since $a_i\ne0$, we obtain from
$$
N-1=\sum\limits_{j\ne i} {\frac{c-a_ia_j}{c-a_ia_j}}=c\,Q_i-a_iR_i=N-1-a_iR_i,
$$
that $R_i =0$. Along with $a_i=-(a_1+...+a_{i-1}+a_{i+1}+...a_N)$ we derive the following equality 
$$
 a_i=
 (c-a_i^2)
 \sum\limits_{j\ne i} {\frac{a_j}{c-a_ia_j}}-\sum\limits_{j\ne i} {a_j}= a_i \sum\limits_{j\not= i} {\frac{a_j^2-a_ja_i}{c-a_ia_j}}.
$$
As $a_i \not= 0$ this yields
\begin{equation}\label{S_i}
S_i:=\sum\limits_{j\not= i} {\frac{a_j^2-a_ja_i}{c-a_ia_j}}=1,
\end{equation}
and subsequently
$$
N-2=\sum\limits_{j\not= i} {\frac{c-a_ia_j}{c-a_ia_j}}-S_i=\sum\limits_{j\not= i}{\frac{c-a_j^2}{c-a_ia_j}}=T_i
$$
\end{proof}

\begin{lemma} \label{lm4} Suppose  $a_1,\ldots,a_N$  are as in Theorem \ref{thmA}. Then 
$$
|a_i| < \sqrt{c}, \quad i=1,...,N. 
$$

\end{lemma}
\begin{proof}  
Let $i>1$. By \eqref{S_i} we have
$$
1=\sum\limits_{j\not= i} {\frac{a_j^2-a_ia_j}{c-a_ia_j}}= \frac{a_1^2-a_ia_1}{c-a_ia_1}+\sum\limits_{2\leq j\not= i} {\frac{a_j^2-a_ia_j}{c-a_ia_j}}.
$$
Then 
$$
\sum\limits_{2\leq j\not= i} {\frac{a_j^2-a_ia_j}{c-a_ia_j}}= {\frac{c-a_1^2}{c-a_ia_1}},\quad i=2,\ldots,N. 
$$
Therefore, 
$$
\sum\limits_{i=2}^N\sum\limits_{2\leq j\not= i} {\frac{a_j^2-a_ia_j}{c-a_ia_j}}= \sum\limits_{i>j=2}^N {\frac{(a_i-a_j)^2}{c-a_ia_j}}=(c-a_1^2)\sum\limits_{i=2}^N {\frac{1}{c-a_ia_1}}=(c-a_1^2)Q_1.
$$
Since $Q_1=N$ and by the assumption $c-a_ia_j>0$,  we have
\begin{equation}\label{Positivity}
c-a_1^2=\frac{1}{N}\sum\limits_{i>j=2}^N {\frac{(a_i-a_j)^2}{c-a_ia_j}}>0. 
\end{equation}
We may assume that   $|a_1|\ge |a_i|$ for all $i$.  Thus, \eqref{Positivity} implies that $c-a_i^2>0$. \end{proof} 

\medskip

\noindent{\it Proof of Theorem \ref{thmA}}: 
Let 
$$
F(t):=\sum\limits_{j=1}^N{\frac{c-a_j^2}{c-ta_j}}. 
$$
Then Lemma \ref{lm3} implies that for all $i=1,\ldots,N$ 
\begin{equation}\label{ConvexValues}
F(a_i)=N-1. 
\end{equation}
Since 
$$
F''(t)=2\sum\limits_{j}{\frac{\left(c-a_j^2\right)a_j^2}{(c-ta_j)^3}},
$$
by Lemma \ref{lm4} we have $F''(t)>0$ for $t\in (-\sqrt{c},\sqrt{c})$.
Hence $F(t)$ is a convex function in this interval. Therefore, the equation $F(t)=N-1$ has at most two solutions. By assumptions we have $a_i>0$ for $i=1,\dots , k$ and $a_i<0$, for $i=k+1,\dots , N$. Thus, \eqref{ConvexValues} yields that  all positive $a_i$ are equal and all negative $a_i$ are equal too.

\section{Stationary Configurations - Proofs}

We are now in a position to prove the classification result Theorem \ref{thmMain}.

\medskip

\noindent{\it Proof of Theorem \ref{thmMain}}:
As there is no point that is equidistant from all others we have $a_i\not=0$ for all $i=1,\dots , N$ Theorem \ref{thmA} yields
$$
a:=a_1=\ldots=a_k> 0>  a_{k+1}=\ldots=a_N=:b, 
$$
where $ka+(N-k)b=0$.
As 
$$
a(x_1+\cdots + x_k)+b(x_{k+1}+\cdots +x_N)=0\quad \mbox{and}\quad  x_1+\cdots+x_N=0, 
$$
we obtain that $x_1+\cdots + x_k=0=x_{k+1}+\cdots +x_N$. Moreover, using \eqref{S_i} we easily obtain that $a^2=(N-k)/(kN)$, $b^2=k/((N-k)N)$, and $ab=-1/N$. This yields that $x_i\cdot x_j=-1/(k-1)$ for $1\leq i<j\leq k$, $x_i\cdot x_j=-1/(N-k-1)$ for $k+1\leq i<j\leq N$, and $x_i \cdot x_j =0$ for $1\leq i\leq k<j\leq N$. This proves the theorem.
\hfill $\Box$

\medskip

We next derive that degenerate stationary configurations may not be local minima of the $h$-energy for convex potential interaction $h$.

\medskip

We shall first introduce the following lemma.
 \begin{lemma}\label{Decr_Incr_Lem}
 Let $h:[-1,1]\to \mathbb{R}$ be a strictly convex function and let $a,b\in \mathbb{R}$ be such that $|a|+|b|\leq 1$, $b\not=0$. Then the function 
 $$
 F(t):=h(a+bt)+h(a-bt)
$$
is strictly decreasing for $t\in [-1,0]$ and strictly increasing for $t\in [0,1]$.
 \end{lemma}
 \begin{proof}
 Since $F(t)$ is even, we consider only $t\in [0,1]$. Let $0\leq t_1<t_2\leq 1$. Define
 $$
 \alpha:=\frac{t_1+t_2}{2t_2},\quad \beta:=\frac{t_2-t_1}{2t_2}.
 $$
 Clearly, $\alpha, \beta >0$ and $\alpha+\beta=1$. Observe that 
 $$
 a+bt_1=\alpha(a+bt_2)+\beta(a-bt_2),\quad a-bt_1=\beta(a+bt_2)+\alpha(a-bt_2).
 $$
 Using the strict convexity of $h$ and that $a+bt_2\not= a-bt_2$ ($b\not=0$) we obtain
 \begin{equation}\label{Jenssen}
 h(a+bt_1)<\,\alpha h(a+bt_2)+\beta h(a-bt_2) ,\quad
 h(a-bt_1)<\,\beta h(a+bt_2)+\alpha h(a-bt_2)
 \end{equation}
 Adding the two inequalities in \eqref{Jenssen} we derive the lemma.
\end{proof}


\medskip

\noindent{\it Proof of Theorem \ref{thmDeg}}: 
As $X$ is degenerate, we may assume without loss of generality that the Equatorial hyperplane contains $X$, or $X\subset \{x_d=0\}$. Since $N\geq d+2$, $X$ is not a regular simplex and therefore there are at least two adjacent edges of distinct length, say $|x_3-x_1|\not=|x_3-x_2|$, or equivalently $x_1\cdot x_3\not= x_2\cdot x_3$. Without loss of generality assume 
$$x_1=(r,\sqrt{1-r^2},0,\dots,0), x_2=(r,-\sqrt{1-r^2},0,\dots,0), x_j=(c_{j1},c_{j2},c_{j3},\dots,0),\ j=3,\dots,N,$$ where at least $c_{32}\not=0$. Form the configuration $\widetilde{X}$ with the first two points perturbed 
$$\tilde{x}_1=(r,\sqrt{1-r^2}\cos \theta,0,\dots,\sqrt{1-r^2}\sin \theta), \tilde{x}_2=(r,-\sqrt{1-r^2}\cos \theta,0,\dots,-\sqrt{1-r^2}\sin \theta).$$
Observe that 
$$\tilde{x}_1\cdot x_j=c_{j1}r+c_{j2}\sqrt{1-r^2}\cos \theta, \quad \tilde{x}_2\cdot x_j=c_{j1}r-c_{j2}\sqrt{1-r^2}\cos \theta.$$

We now apply Lemma \ref{Decr_Incr_Lem} with $a=c_{j,1}r$, $b=c_{j2}\sqrt{1-r^2}$, and $t=\cos{\theta}$ to conclude that for all $j$ such that $c_{j2}\not=0$ (this is not empty as $c_{32}\not=0$)
$$
h(\tilde{x}_1\cdot x_j)+h(\tilde{x}_2\cdot x_j)<h({x}_1\cdot x_j)+h({x}_2\cdot x_j).
$$
Obviously if $c_{j2}=0$ we have equality in the above inequality. This implies that $E_h (\widetilde{X}) < E_h (X)$ for all $0<\theta < \pi$. 
\hfill $\Box$

\medskip
%
%
%

\medskip

 \noindent{\it Proof of Theorem  \ref{thmPyramid} }: Theorem \ref{thmMain} shows that non-degenerate stationary configuration $X$ must either split into two orthogonal regular simplexes $X=X_m \cup X_n$ with $m+n=d+2$, or have a vertex that is equidistant to all other vertices. The first case will be dealt with in Section 5.

 Suppose that the second case holds. As in the discussion before the formulation of the theorem, suppose $x_N \cdot x_i=-1/(N-1)$ for all $i=1,\dots,N-1$. For all $i=1,\dots,N-1$ denote $x_i=(y_i,-1/(N-1))$ and let $z_i:=(N-1)y_i/\sqrt{N(N-2)}$. Then $\{z_i \}_{i=1}^{N-1} \subset \mathbb{S}^{d-2}$ satisfy similar force equations as \eqref{VecEqs}. 
 
 As $\{x_i\}$ is non-degenerate, so is $\{z_i\}$. Thus, we have reduced the problem's  dimension. The process will stop and at the last step we shall obtain two orthogonal simplexes. 
 
 So, without loss of generality we may assume the process has stopped after one step, namely we have a configuration of the type $\{ 1,k,m\}$, where one of the points $p:=(0_{k-1},0_{m-1},1)$ is equidistant to all others, and these other points form two regular orthogonal simplexes 
 \[ Y:=\{ (\sqrt{1-1/(k+m)^2}\, y_i,0_{m-1},-1/(k+m))\},\ Z:=\{ (0_{k-1},\sqrt{1-1/(k+m)^2}\, z_j,-1/(k+m))\}\] 
 with $k$ and $m$ points respectively (here $1+k+m=d+2$). We perturb the configuration $X:=\{p,Y,Z\}$ to $\widetilde{X}_t:=\{p,\widetilde{Y}_t,\widetilde{Z}_t\}$, where
 $$
 \widetilde{Y}_t= \left\{ \left( \sqrt{1-(mt+1/(k+m))^2}\, y_i,0_{m-1},-1/(k+m)-mt \right) \right\}_{i=1}^k 
 $$
 and
 $$
  \widetilde{Z}_t= \left\{ \left(0_{k-1},\sqrt{1-(kt-1/(k+m))^2}\, z_j,-1/(k+m)+kt \right) \right\}_{j=1}^m.
 $$
The logarithmic energy of the perturbed configuration as a function of $t$ is given by 

\begin{align}
E_{\log}(X_t)=&\frac{k(k+1)}{2}\log \left( \frac{1}{1+\frac{1}{k+m}+mt}\right)+\frac{k(k-1)}{2}\log\left( \frac{1}{1-\frac{1}{k+m}-mt} \cdot \frac{k}{k-1} \right)\nonumber\\
& +\frac{m(m+1)}{2}\log \left(\frac{1}{1+\frac{1}{k+m}-kt}\right) +\frac{m(m-1)}{2}\log \left(\frac{1}{1-\frac{1}{k+m}+kt}\cdot \frac{m}{m-1}\right)\\
 &+km\log \left(\frac{1}{1-(\frac{1}{k+m}+mt)(\frac{1}{k+m}-kt)}\right) =: f(t) \nonumber
\end{align}

The derivative can be computed as
\begin{align}
f^\prime (t)=&\frac{km(m+k)t(mt+\frac{1}{k+m})(kt-\frac{1}{k+m})}{1-(\frac{1}{k+m}+mt)(\frac{1}{k+m}-kt)}
\left[ \frac{m}{1-\left( \frac{1}{k+m}+mt \right)^2} +\frac{k}{1-\left(\frac{1}{k+m}-kt\right)^2} \right] . 
\end{align}
Observe that the denominator of the first fraction and the expression in the brackets are positive as $\widetilde{X}_t \subset \mathbb{S}^{d-1}$. Therefore, 
\[  {\rm sign} \left( f^\prime (t) \right)=  {\rm sign} \left( t  \left( mt+\frac{1}{k+m} \right) \left( kt-\frac{1}{k+m} \right) \right).\] 

Thus, we observe that for $t\in [-1/m(k+m),0]$ the logarithmic energy is strictly increasing and for $t\in [0,1/k(k+m)]$ it is strictly decreasing, thus being maximal when $t=0$. This shows that $\{1,k,m\}$ is not a local minimum and we can make a continuous perturbation that decreases the energy from $t=0$ to $t=-1/m(k+m)$, which corresponds to a $\{k, m+1\}$ configuration of two orthogonal simplexes, or to $t=1/k(k+m)$, which corresponds to a $\{k+1, m\}$ configuration.

Of course, should we consider one of the simplexes, say $Y$, fixed and vary the other one within the hyperplne in which it is embedded (which is equivalent to let $z_j$ vary), then the maximum is attained when $Z$ is regular. Therefore, this is a case of a saddle point for the logarithmic energy.
\hfill $\Box$
\medskip

\section{Local Minima - Proof of the Main Result}

The proof of Theorem \ref{thmLogOpt} utilizes the following two lemmas.

\begin{lemma}\label{MatrixIneq1}
Let $A=(a_{ij} )$ be an $m\times m $ matrix, $m\geq 3$, such that {\rm (a)} $a_{ii}=0$, $i=1,\dots,m$; and {\rm (b)} $\sum_{j=1}^m a_{ij} =0$. Then the following inequality holds
\begin{equation}\label{D1'}
\sum_{1\leq i<j\leq m}
(a_{ij}+a_{ji})^2 \geq \frac{1}{m-2}\sum_{j=1}^m x_j^2,\quad {\rm where}\quad x_j:=\sum_{i=1}^m a_{ij} . 
\end{equation}
\end{lemma}
\begin{proof} For all $i,j=1,\dots,m$ define 
\[\beta_{ij}:=\frac{1}{m^2-2m} x_i+\frac{m-1}{m^2-2m} x_j, \quad i\not= j, \quad {\rm and}\ \  \beta_{ii}=0.\] 
Since $\sum_{j=1}^m x_j =0$, we have $\sum_{j=1}^m \beta_{ij}=0$ and $\sum_{i=1}^m \beta_{ij}=x_j$, i.e. 
\[ \sum_{j=1}^m \beta_{ij}=\sum_{j=1}^m a_{ij}\quad {\rm and} \quad \sum_{i=1}^m \beta_{ij}=\sum_{i=1}^m a_{ij}.\] 
Let $\widetilde{a}_{ij}:=a_{ij}-\beta_{ij}$. Then \[\sum_i \widetilde{a}_{ij}=\sum_j \widetilde{a}_{ij}=0.\]
Consider $t_{ij}:=a_{ij}+a_{ji}=w_{ij}+\beta_{ij}+\beta_{ji}$, where $w_{ij}=\widetilde{a}_{ij}+\widetilde{a}_{ji}$. Then $t_{ij}=w_{ij}+\frac{x_i}{m-2}+\frac{x_j}{m-2}$, $i\not= j$, where $\sum_i w_{ij}=\sum_j w_{ij}=0$ (observe that $t_{ii}=0$).
Then 
\[ \sum_{i<j} t_{ij}^2=\sum_{i<j} \left( w_{ij}+\frac{x_i}{m-2}+\frac{x_j}{m-2}\right)^2 =\sum_{i<j}w_{ij}^2+ \frac{1}{m-2}\sum_{i=1}^m x_i^2,\]
which implies \eqref{D1'}. 
\end{proof}

\medskip

\begin{lemma} \label{MatrixIneq2} 
Given an $m\times n$ matrix $F=(f_{ij})$ and an $n\times m$ matrix $G=(g_{ij})$ such that ${\sum_{j=1}^n f_{ij}=0}$ for all $i=1,\dots,m$ and  $\sum_{j=1}^m g_{ij}=0$ for all $i=1,\dots,n$. Then we have
\begin{equation} \label{D3'} \sum_{i=1}^n \sum_{j=1}^m (f_{ij}+g_{ji})^2\geq \frac{1}{m}\sum_{j=1}^n y_j^2+\frac{1}{n} \sum_{i=1}^m z_i^2, \quad {\rm where}\quad  y_j:=\sum_{i=1}^m f_{ij}, z_i:=\sum_{j=1}^n g_{ji}.
\end{equation}
\end{lemma}
\begin{proof} Let 
\[ \widetilde{f}_{ij}:=f_{ij}-\frac{y_j}{m} \quad {\rm and}\quad \widetilde{g}_{ij}:=g_{ij}-\frac{z_i}{n}.\] Since $\sum_j y_j =\sum_i z_i=0$, we have $\sum_{i,j} (\widetilde{f}_{ij}+\widetilde{g}_{ji})=0$.
Let $t_{ij}:=\widetilde{f}_{ij}+\widetilde{g}_{ji}$. Observe that 
\[\sum_{i=1}^m t_{ij}=\sum_{j=1}^n t_{ij}=0.\] 
From
\[ f_{ij}+g_{ji}=\frac{y_j}{m}+\frac{z_i}{n}+t_{ij}.\]
one derives that
\[\sum_{i=1}^m \sum_{j=1}^n (f_{ij}+g_{ji})^2=\sum_{i=1}^m \sum_{j=1}^n \left( \frac{y_j}{m}+\frac{z_i}{n}+t_{ij}\right)^2 =\sum_{i=1}^m \sum_{j=1}^n t_{ij}^2+\frac{1}{m}\sum_{j=1}^n y_j^2+\frac{1}{n}\sum_{i=1}^m z_i^2,\]
which completes the proof. \end{proof}

 \noindent{\it Proof of Theorem  \ref{thmLogOpt} }: Denote the two regular orthogonal simplexes, whose centers of mass are both in the origin with
 \[X_m=\{ {x}_1,  {x}_2, \dots,  {x}_m\},\quad X_n=\{ {x}_{m+1}, {x}_{m+2}, \dots,  {x}_{m+n}\}. \]
 Let $\epsilon>0$ be a positive number and let us perturb the points of the simplexes to $ {y}_i \in \mathbb{S}^{d-1}$, $ {y}_i:= {x}_i+ {h}_i$, where $\|  {h}_i\|<\epsilon$, $i=1,\dots d+2$. Denote the new configuration $Y=Y_m\cup Y_n$. Since $\| {x}_i \|=\| {y}_i \|=1$, we have $2 {x}_i\cdot  {h}_i =-\| {h}_i \|^2$. We also have $1-y_i\cdot y_j=(1-x_i \cdot x_j)(1-z_{i,j})$, where
 \begin{equation} \label{zij}
 z_{i,j}:=
 \begin{cases} 
\displaystyle{ \frac{m-1}{m}(x_i\cdot h_j+x_j\cdot h_i+h_i\cdot h_j) }, & 1\leq i\not= j\leq m\\
 & \\
 x_i \cdot h_j + x_j \cdot h_i + h_i \cdot h_j , & i \leq m< j {\rm \ or\  } j \leq m< i \\
  & \\
\displaystyle{  \frac{n-1}{n}(x_i\cdot h_j+x_j\cdot h_i+h_i\cdot h_j) }, & m< i \not= j\leq m+n.
 \end{cases}
 \end{equation}
 Clearly $|z_{i,j}|<2\epsilon + O(\epsilon^2)$. The definition of the logarithmic energy \eqref{LogEn} implies that
 \begin{equation}\label{EnDiff}
 2\left[ E_{\rm log} (Y)-E_{\rm log}(X) \right] =-\sum_{1\leq i\not= j \leq m+n} \log (1-z_{i,j}) =\sum_{1\leq i\not= j \leq m+n} \left( z_{i,j}+\frac{z_{i,j}^2}{2}\right)+O (\epsilon^3).
 \end{equation}
 Excluding $O(\epsilon^3 )$ terms from \eqref{EnDiff} the remainder is 
 \[D:=\sum_{1\leq i\not= j \leq m+n} z_{i,j}+\frac{1}{2}\sum_{1\leq i\not= j\leq m+n} \left( \frac{x_i \cdot h_j+x_j \cdot h_i}{1-x_ \cdot x_j}\right)^2.\]
To compute $D$, without loss of generality we may assume that $x_i=(p_i,0)$, $h_i=(a_i,b_i)$, $i=1,\dots,m$ and $x_{m+j} = (0,q_j )$, $h_{m+j}=(c_j,d_j)$, $j=1,\dots,n$, where $p_i, a_i, c_j \in \mathbb{R}^{m-1}$ and $q_j, b_i, d_j \in \mathbb{R}^{n-1}$. 
Application of $2 {x}_i\cdot  {h}_i =-\| {h}_i \|^2$ yields
\[ \sum_{1\leq i\not= j \leq m}  z_{i,j} = \frac{2(m-1)}{m}\left( \sum_{i=1}^m x_i \right)\left( \sum_{i=1}^m h_i \right)+\Big\|\sum_{i=1}^m h_i\Big\|^2-\frac{1}{m}\Big\|\sum_{i=1}^m h_i\Big\|^2 .\]
As the origin is the center of mass of $X_m$ we have 
\[ \sum_{1\leq i\not= j \leq m}  z_{i,j} = \Big\|\sum_{i=1}^m h_i \Big\|^2 -\frac{1}{m}\left( \Big\|\sum_{i=1}^m a_i\Big\|^2+ \Big\|\sum_{i=1}^m b_i\Big\|^2\right) .\]
Similarly,
\[ \sum_{m+1\leq i\not= j \leq m+n}  z_{i,j} = \Big\|\sum_{j=m+1}^{m+n} h_j \Big\|^2 -\frac{1}{n}\left( \Big\|\sum_{j=1}^n c_j\Big\|^2+ \Big\|\sum_{j=1}^n d_j\Big\|^2\right) ,\]
and 
\[ \sum_{i=1}^m\sum_{j=m+1}^{m+n}  z_{i,j} = \left( \sum_{i=1}^m h_i \right)\cdot \left( \sum_{i=m+1}^{m+n} h_j \right),\]

This simplifies to
\begin{align}
D=& \Big\|\sum_{i=1}^{m+n} h_i \Big\|^2 - \frac{1}{m}\left( \Big\|\sum_{i=1}^m a_i\Big\|^2+\Big\|\sum_{i=1}^m b_i \Big\|^2\right)- \frac{1}{n}\left( \Big\|\sum_{j=1}^n c_j\Big\|^2+\Big\|\sum_{j=1}^{n} d_j \Big\|^2\right) \nonumber \\
&+\left(\frac{m-1}{m}\right)^2 \sum_{1\leq i<j\leq m} (p_i \cdot a_j +p_j \cdot a_i )^2+\left(\frac{n-1}{n}\right)^2 \sum_{1\leq i<j\leq n} (q_i \cdot d_j +q_j \cdot d_i )^2\\
&+ \sum_{i=1}^m \sum_{j=1}^n (p_i \cdot c_j +q_j \cdot b_i )^2.\nonumber
\end{align}

Thus, in this case we shall reduce the theorem to proving the inequalities
\begin{equation}\label{D1}
D_1 :=\left(\frac{m-1}{m}\right)^2 \sum_{1\leq i<j\leq m} (p_i \cdot a_j +p_j \cdot a_i )^2-\frac{1}{m}\Big\| \sum_{i=1}^m a_i \Big\|^2\geq 0
\end{equation}
\begin{equation}\label{D2}
D_2 :=\left(\frac{n-1}{n}\right)^2 \sum_{1\leq i<j\leq n} (q_i \cdot d_j +q_j \cdot d_i )^2-\frac{1}{n}\Big\| \sum_{j=1}^n d_j \Big\|^2\geq 0
\end{equation}
and 
\begin{equation}\label{D3}
D_3:=\sum_{i=1}^m \sum_{j=1}^n (p_i \cdot c_j +q_j \cdot b_i )^2-\frac{1}{m}\| \sum_{i=1}^m b_i \|^2-\frac{1}{n}\Big\| \sum_{j=1}^n c_j \Big\|^2 \geq 0. 
\end{equation}

{If we denote $\widetilde{h}_i:=h_i - \left( x_i\cdot h_i \right) x_i$, $i=1,\dots, m+n$, then $x_i \cdot \widetilde{h}_i =0$.
Since $2x_i \cdot h_i = -\|h_i\|^2$, we respectively have $\widetilde{a}_i = a_i+O(\epsilon^2)p_i$, $\widetilde{b}_i=b_i$, for $i=1,\dots, m$, and $\widetilde{c}_j=c_j$ and $\widetilde{d}_j = d_j+O(\epsilon^2)q_j$ for $j=1,\dots, n$. Therefore, by adding additional $O(\epsilon^3)$ terms to \eqref{EnDiff},  it suffices to prove (18) and (19) under the additional assumption that $p_i \cdot a_i =0$ and $q_j\cdot d_j =0$. }

To prove the inequalities we embed the first simplex $X_m=\{p_1,\dots, p_m \}$ in the hyperplane of $\mathbb{R}^m$ that is orthogonal to $(1,1,\dots,1)$. Similarly, we embed the second simplex $X_n=\{q_1,\dots, q_n \}$ in $\mathbb{R}^n$. Thus, we embed $X_m\cup X_n \subset \mathbb{R}^m \times \mathbb{R}^n$. Denote $w_m=(\frac{1}{m},\frac{1}{m},\dots, \frac{1}{m})  \in \mathbb{R}^m$ and let $\widetilde{p}_i:=e_i-w_m$, $i=1,\dots,m$. Then $p_i =\sqrt{\frac{m}{m-1}} \, \widetilde{p}_i$. Similarly, if $\widetilde{q}_j:=e_j-w_n$, then $q_j =\sqrt{\frac{n}{n-1}} \, \widetilde{q}_j$. For the perturbation vectors $a_i=(a_{i1},a_{i2},\dots,a_{im})$, $b_i=(b_{i1},b_{i2},\dots,b_{in})$, $c_j=(c_{j1},c_{j2},\dots,c_{jm})$, $d_j=(b_{j1},b_{j2},\dots,b_{jn})$, we will have that $\sum_{j=1}^m a_{ij} =0$, $\sum_{j=1}^n b_{ij} =0$, $i=1,\dots,m$, and $\sum_{j=1}^m c_{ij} =0$, $\sum_{j=1}^n d_{ij} =0$, $i=1,\dots,n$. The conditions $p_i \cdot a_i =0$ and $q_j \cdot d_j =0$ imply that $a_{ii}=0$ for all $i=1,\dots,m$ and $d_{jj}=0$ for all $j=1,\dots,n$.

Using that $\widetilde{p}_i\cdot a_j=a_{ji}$ we can re-write \eqref{D1} as
\[
\sum_{1\leq i<j\leq m}
(a_{ij}+a_{ji})^2 \geq \frac{1}{m-1}\sum_{j=1}^m \left( \sum_{i=1}^m a_{ij} \right)^2, 
\]
which follows from the stronger inequality \eqref{D1'} in Lemma \ref{MatrixIneq1}. Observe that equality holds in \eqref{D1} and \eqref{D2} if and only $a_{ij}+a_{ji}=0$ and $d_{ij}+d_{ji}=0$ respectively, which is equivalent to $p_i\cdot a_j + p_j \cdot a_i=0$, $q_j\cdot d_i+ q_i \cdot d_j=0$, $\sum a_i=0$, and $\sum d_j =0$.

In a similar manner we shall utilize Lemma \ref{MatrixIneq2} to derive the inequality \eqref{D3}. We have that 
\[p_i\cdot c_j +q_j \cdot b_i=\sqrt{\frac{m}{m-1}} c_{ji}+\sqrt{\frac{n}{n-1}} b_{ij} \]
with the substitution $f_{ij}=\sqrt{\frac{n}{n-1}} b_{ij} $ and $g_{ji}=\sqrt{\frac{m}{m-1}} c_{ji}$ we re-write \eqref{D3} as
\[  \sum_{i=1}^n \sum_{j=1}^m (f_{ij}+g_{ij})^2\geq \frac{1}{m}\frac{n-1}{n}\sum_{j=1}^n y_j^2+\frac{1}{n} \frac{m-1}{m}\sum_{i=1}^m z_i^2,\]
which clearly follows from \eqref{D3'}. Moreover, equality occurs if and only if $p_i\cdot c_j + q_j \cdot b_i=0$, $\sum c_i=0$, and $\sum b_j =0$.

\medskip 

To summarize, the quadratic term in $\epsilon$ will be strictly positive, and hence $E_{\rm log} (Y)-E_{\rm log}(X)>0$, for any perturbation vectors $\{a_i, b_i, c_i, d_i\}$ ($p_i\cdot a_i=0$, $q_j\cdot d_j=0$), except when 
\[ p_i\cdot a_j + p_j \cdot a_i=0, \quad q_j\cdot d_i+ q_i \cdot d_j=0, \quad p_i\cdot c_j + q_j \cdot b_i=0,\] 
and 
\[ \sum_{i=1}^m a_i=0, \quad \sum_{i=1}^m b_i =0, \sum_{j=1}^n c_i=0,\quad \sum_{j=1}^n d_j =0.\] 

Utilizing \eqref{zij} and these conditions, one simplifies \eqref{EnDiff} to

\begin{align*}\label{EnDiff4}
 2\left[ E_{\rm log} (Y)-E_{\rm log}(X) \right] = & \frac{(m-1)^2}{2m^2}\sum_{1\leq i\not= j \leq m}  (a_i\cdot a_j)^2+ \frac{(n-1)^2}{2n^2}\sum_{1\leq i\not= j \leq n} (d_i\cdot d_j)^2\\
 & +\sum_{i=1}^m\sum_{j=1}^n (b_i\cdot c_j )^2+O (\epsilon^5).
\end{align*}
Clearly, the quartic term will be positive, unless all inner products vanish, in which case we easily derive that $a_i=c_j=0$ and $b_i=d_j=0$ for all $i=1,\dots,m$ and $j=1,\dots,n$.
This completes the proof. \hfill $\Box$

\section{Concluding remarks and open problems}

\subsection{Morse theory of $(d+2)$-configurations in $d$--dimensions}

Let $\conf(d,N)$ denote  the configuration space of $N$--tuples of points in ${\mathbb S}^{d-1}$ up to isometry. Then 
$$
\conf(d,N)=({\mathbb S}^{d-1}\times...\times{\mathbb S}^{d-1})/SO(d) 
$$
and the dimension of this space is 
$$
\dim{\conf(d,N)}=(d-1)N-\frac{d(d-1)}{2}. 
$$
In particular, 
$$
\dim{\conf(d,d+2)}=\frac{(d+2)(d-1)}{2}, \quad \dim{\conf(3,5)}=7. 
$$

The {\em Morse index} of a critical point $x$ of a smooth function $f$ on a manifold $M$ is equal, by definition, to the negative index of inertia of the Hessian of $f$ at $x$.  

By the above results we have a classification of all critical (stationary) points of $E_{log}$ on $\conf(d,d+2)$. In particular, if $d=3$ then we have only three types of critical points: $C_0$ of type (0,5), $C_1$ of type (1,2,2) and $C_2$ of type (2,3). 

\begin{theorem} The Morse index of $E_{log}$ on $\conf(3,5)$ at $C_i$, $i=0,1,2$, is $2-i$. 
\end{theorem}
\begin{proof} Denote by $\mu(x)$ the Morse index of $E_{log}$ on $\conf(3,5)$ at $x$. Since $E_{log}$ has at $C_2$ a minimum, $$\mu(C_2)=0.$$

Let $x, y \in {\mathbb S}^{2}$  with spherical coordinates $(\phi,\theta)$ and $(\phi',\theta')$. Then 
$$
[x-y|^2=2-2(\sin{\theta}\sin{\theta'}\cos(\phi-\phi')+\cos{\theta}\cos{\theta'}).  \eqno (21)
$$

Consider $X=\{x_1,...,x_5\} \subset {\mathbb S}^{2}$ with spherical coordinates $\{(\phi_1,\theta_1),...,(\phi_5,\theta_5)\}$ as a point in $\conf(3,5)$. Without loss of generality we can assume that  $\phi_1=\theta_1=\phi_2=0$. Then 
$$
v=v(X):=(\theta_2,\phi_3,\theta_3,\phi_4,\theta_4,\phi_5,\theta_5)
$$
is a vector of seven variables that uniquely defined a point in the configuration space $\conf(3,5)$. It is not hard to show that 
$$
v(C_0)=(2\pi/5,0,4\pi/5,\pi,4\pi/5,\pi,2\pi/5), \quad v(C_1)=(w,\pi/2,w,\pi,w,3\pi/2,w), \; w:=\arccos(-1/4). 
$$

Using (21) we can represent  $E_{log}(X)$ as a function $f(v)$. Then the Hessian $H(v)$ of $f(v)$ and its eigenvalues at $v(C_0)$ and $v(C_1)$ can be found by direct calculations.  These calculations show that 
$$\mu(C_0)=2, \quad \mu(C_1)=1.$$
\end{proof}

It is an interesting problem: 

\medskip

{\em Find the Morse indexes of all critical points of $E_{log}$ on $\conf(d,d+2)$ for all $d$. }

\subsection{Extensions of the main theorems for other potentials}

It is of interest to determine other potentials for which we are able to characterize the optimal $d+2$-point configurations on $\mathbb{S}^{d-1}$. 

\medskip 

(a) {\bf Riesz potentials:}  It is an interesting open problem to find all critical configurations of  the energy $E_h$ on $\conf(3,5)$ for Riesz potential interaction $h(t)=(1-t)^{-s/2}$ (see \cite[Section 2.5]{BHS} for details). Even in this simple case of five points on $\mathbb{S}^2$, rigorous results are limited. Here we have two competing configurations: the triangular bi-pyramid (TBP) consisting of one point at the north pole, one at the south pole, and three arranged in an equilateral triangle around the equator; and the regular four-pyramid (FP) with square base and varying height on the parameter $s$.

For $s=0$ it is shown in \cite{DLT} that TBP is the unique up to rotations minimizer of $E_{\rm log}$. Utilizing computer aided proofs the optimality of TBP is established for $s=-1$ in \cite{HoS} and for $s=1,2$ in \cite{Sch}, which was recently extended \cite{Sch2} to  show that there is a constant $s_* \approx 15.048081$ (conjectured in \cite{MKS}), such that the TBP is the global minimizer for $s\leq s_*$, and for $s_*\leq s<15+\frac{25}{512}$ the FP is the global minimizer. 

The determination of $d+2$ points on $\mathbb{S}^{d-1}$ with minimal Riesz energy is an interesting open problem, even for $d=4$ and particular values of $s$, say the Newton potential interaction case $s=2$. We expect similar transition value $s_*(d)$ so that for $0\leq s\leq s_*(d)$ the optimal configuration will be the configuration consisting of two orthogonal simplexes of minimal cardinality difference.

\medskip 

(b) {\bf Bi-quadratic potential energy:} In \cite{T} it was shown that the TBP
is the unique up to rotations optimal spherical configuration of five points on $\mathbb{S}^2$ for the bi-quadratic potential 
\[h(x_i \cdot x_j)=a\, (x_i\cdot x_j )^2+b\, (x_i \cdot x_j )+c,\quad a>0, \ b>2a.\]

We expect that our results will extend to bi-quadratic potentials and $d+2$ points on $\mathbb{S}^{d-1}$ and intend to return to this problem in the near future.
\medskip 

(c) {\bf Optimal  1--designs of cardinality $d+2$ in $d$--dimensions:} 

Since log-optimal stationary configurations $X$ have their centroid at the origin, It is of interest to minimize various energies among the class of configurations that are $1$-designs, i.e. $x_1+...+x_N=0$, also called balanced configurations. Minimizing energy over balanced configurations is an interesting problem that has physical meaning, we intend to return to it in the near future as well.

\medskip

(d) {\bf SDP bounds for optimal configurations:} 

Recently, in   \cite{Mu19} have been obtained new SDP bounds for distance distribution and distance graphs of spherical codes. It is an interesting problem to extend these bounds for optimal configurations. 


\subsection{Optimal $(d+k)$-configurations in $d$--dimensions}

Now we consider $X\subset {\Bbb S}^{d-1}$ with $d+2\le |X| \le 2d$.  Rankin's theorem states that   if $X$ is a subset of ${\mathbb S}^{d-1}$ with $|X|\ge d+2$, then the minimum distance between points in $X$ is at most $\sqrt{2}$. For the case $|X|=2d$ Rankin proved that $X$ is a regular cross--polytope. 
 Wlodzimierz Kuperberg \cite{K} extended Rankin's theorem. 
 
 \medskip
 
\noindent {\bf Kuperberg's theorem:} {\em Let  $X$ be a
 ($d + k)$--point subset of  ${\mathbb S}^{d-1}$ with  $2\le k \le d$ such that the minimum distance between points is at least $\sqrt{2}$. Then  ${\Bbb R}^d$ splits into the orthogonal product $\prod_{i=1}^k{L_i}$ of nondegenerate linear subspaces $L_i$ such that for  $S_i:=X\cap L_i$ we have $|S_i|=d_i+1$ and $\rank(S_i)=d_i$  $(i = 1, 2, . . . , k)$, where  $d_i:= \dim{L_i}$.}
 
  \medskip

The following theorem is equivalent to \cite[Theorem 4.2]{Mu20}.

\begin{theorem}\label{th62} Let $h:[-1,1)\to{\mathbb R}$ be a convex monotone increasing function. Let  $X$ be a
subset of ${\mathbb S}^{d-1}$ of cardinality  $d + k$ with  $2\le k \le d$ such that the minimum distance between distinct points in $X$ is at least $\sqrt{2}$. Then the set of all local minima of $E_h$ (see (2) in Sect. 1) consists of $k$ orthogonal to each other regular $d_i$--simplexes $S_i$  such that all $d_i\ge1$ and $d_1+...+d_k=d$.
\end{theorem}

Actually, this theorem easily follows from the optimality of simplices (\cite[Theorem 4.1]{Mu20}) and Kuperberg's theorem. 

\medskip

Let $h(t):=-\log(1-t)$.  If $k=2$, then Theorem 1.1 yields Theorem \ref{th62}. Moreover, we don't need the assumption that for all $x,y\in X$ with $x\ne y$ we have  $|x-y|\ge\sqrt{2}$. 
For the case $k=d$ it is proven by \cite{KY}  that Log--optimal X is a regular cross--polytope. 

 It is an interesting open problem to extend Theorem 1.1 for $2<k\le d$. Our conjecture is that for all $k$ such that $2\le k \le d$   we have the same result as in Theorem \ref{th62}.

\medskip

\noindent{\bf Acknowledgment.}  This paper is based upon work supported by the National Science Foundation under Grant No. DMS-1439786 while the authors were in residence at the Institute for Computational and Experimental Research in Mathematics in Providence, RI, during the Spring 2018 semester. The research of the first author was supported, in part, by a Simons Foundation grant no. 282207, and in part,
by the U. S. National Science Foundation under grant DMS-1936543.

\medskip

P. D. Dragnev, Purdue University Fort Wayne, Department of Mathematical Sciences, Fort Wayne, IN 46805, USA

 {\it E-mail address:} dragnevp@pfw.edu
\medskip

 O. R. Musin,  University of Texas Rio Grande Valley, School of Mathematical and
 Statistical Sciences, One West University Boulevard, Brownsville, TX 78520, USA \& \\

 {\it E-mail address:} oleg.musin@utrgv.edu

\end{document}